\documentclass[12pt]{amsart}
\usepackage{ifthen,verbatim}
\usepackage{floatflt,graphicx,graphics}
\usepackage{subfig}

 \usepackage{tikz} 
\usepackage{mathrsfs}
\usepackage{amsmath,amssymb,amsthm}

\numberwithin{equation}{section}
\nonstopmode
\numberwithin{equation}{section}
\theoremstyle{plain}
\newtheorem{theorem}[equation]{Theorem}
\newtheorem{conjecture}[equation]{Conjecture}
\newtheorem{lemma}[equation]{Lemma}
\newtheorem{corollary}[equation]{Corollary}

\newtheorem{proposition}[equation]{Proposition}

\theoremstyle{definition}
\newtheorem{definition}[equation]{Definition}

\newtheorem{remark}[equation]{Remark}

\theoremstyle{remark}

\newcommand{\R}{\mathbb{R}}

\newcommand{\C}{\mathbb{C}}

\newcommand{\B}{\mathbb{B}}

\newcommand{\uhp}{\mathbb{H}}

{\qed\bigskip}

\newcounter{alphabet}

\newcounter{minutes}
\divide\time by 60
\newcounter{hours}
\multiply\time by 60
\addtocounter{minutes}{-\time}

\begin{document}
\bibliographystyle{amsplain}
\title[Introducing a new intrinsic metric]
{
Introducing a new intrinsic metric
}

\def\thefootnote{}
\footnotetext{
\texttt{\tiny File:~\jobname .tex,
          printed: \number\year-\number\month-\number\day,
          \thehours.\ifnum\theminutes<10{0}\fi\theminutes}
}
\makeatletter\def\thefootnote{\@arabic\c@footnote}\makeatother

\author[O. Rainio]{Oona Rainio}
\address{Department of Mathematics and Statistics, University of Turku, FI-20014 Turku, Finland}
\email{ormrai@utu.fi}
\author[M. Vuorinen]{Matti Vuorinen}
\address{Department of Mathematics and Statistics, University of Turku, FI-20014 Turku, Finland}
\email{vuorinen@utu.fi}

\begin{abstract}
A new intrinsic metric called $t$-metric is introduced. Several sharp inequalities between this metric and the most common hyperbolic type metrics are proven for various domains $G\subsetneq\R^n$. The behaviour of the new metric is also studied under a few examples of conformal and quasiconformal mappings, and the differences between the balls drawn with all the metrics considered are compared by both graphical and analytical means.  
\end{abstract}
\maketitle

\section{Introduction}

In geometric function theory, one of the topics studied deals with the variation of geometric entities such as distances, ratios of distances, local geometry and measures of sets under different mappings. For such studies, we need an appropriate notion of distance that is compatible with the class of mappings studied. In classical function theory of the complex plane, one of the key concepts is the hyperbolic distance, which measures not only how close the points are to each other but also how they are located inside the domain with respect to its boundary.

The hyperbolic distance also serves as a model when we need generalisations to subdomains $G$ of arbitrary metric spaces $X$. These generalized distances behave like the hyperbolic metric in the aspect that they define the Euclidean topology and, in particular, we can cover compact subsets of $G$ using balls of the generalized metrics. Thus, the boundary of the domain has a strong influence on the inner geometry of the domain defined by some chosen metric.

Since the classical hyperbolic geometry acts as a model, some of its key features are inherited by the generalizations but not all. For instance, it is desirable to study local behaviour of functions and we need to have a metric that is locally comparable with the Euclidean geometry. Such a metric is here called an \emph{intrinsic metric}. Note that there is no established definition for this concept and it is sometimes required, for instance, that the closures of the balls defined with an intrinsic metric never intersect the boundary of the domain. An example of an intrinsic metric is the following new metric, on which this work focuses.

\begin{definition}\label{def_1}
Let $G$ be some non-empty, open, proper and connected subset of a metric space $X$. Choose some metric $\eta_G$ defined in the closure of $G$ and denote $\eta_G(x)=\eta_G(x,\partial G)=\inf\{\eta_G(x,z)\text{ }|\text{ }z\in \partial G\}$ for all $x\in G$. The \emph{$t$-metric} for a metric $\eta_G$ in a domain $G$ is a function $t_G:G\times G\to[0,1],$
\begin{align*}
t_G(x,y)=\frac{\eta_G(x,y)}{\eta_G(x,y)+\eta_G(x)+\eta_G(y)},  \end{align*}
for all $x,y\in G$. Here, we mostly focus on the special case where $G\subsetneq\R^n$ and $\eta_G$ is the Euclidean distance.
\end{definition}

Our work in this paper is motivated by the research of several other mathematicians. During the past thirty years, many intrinsic metrics have been introduced and studied \cite{chkv, fmv, hkvbook, imsz, ms}. It is noteworthy that each metric might be used to discover some intricate features of mappings not detected by other metrics. Since our new metric differs slightly from other intrinsic metrics and has a relatively simple definition, it could potentially be a great help for new discoveries about intrinsic geometry of domains. For instance, there is one inequality that is an open question for the triangular ratio metric, but could potentially be proved for the $t$-metric, see Conjecture \ref{conj_tmobius} and Remark \ref{rmk_sbmobius}.

Unlike several other hyperbolic type metrics, such as the triangular ratio metric or the hyperbolic metric itself, the $t$-metric does not have the property about the closed balls never intersecting with the boundary, see Theorem \ref{thm_ttouchesboundary}. This is an interesting aspect for this metric clearly fulfills most of the others, if not all, properties of a hyperbolic type metrics listed in \cite[p. 79]{hkvbook}. Consequently, we have found an intrinsic metric that does not have one of the common properties of hyperbolic type metrics.

In this paper, we will study this new metric and its connection to other metrics. In Section \ref{sect_tbounds}, we prove that the function of Definition \ref{def_1} is really a metric and find the sharp inequalities between this metric and several hyperbolic type metrics, including also the hyperbolic metric, in different domains. In Section \ref{sect_qmaps}, we show how the $t$-metric behaves under certain quasiconformal mappings and find the Lipschitz constants for M\"obius maps between balls and half-spaces. Finally, in Section \ref{sect_metricballs}, we draw $t$-metric disks and compare their certain properties to those of other metric disks.

{\bf Acknowledgements.} The research of the first author was supported by Finnish Concordia Fund.

\section{Preliminaries}

In this section, we will introduce the definitions of a few different metrics and metric balls that will be necessary later on but, first, let us recall the definition of a metric.

\begin{definition}\label{def_metric}
For any non-empty space $G$, a \emph{metric} is a function $\eta_G:G\times G\to[0,\infty)$ that fulfills the following three conditions for all $x,y,z\in G$:\\
\emph{(1)} Positivity: $\eta_G(x,y)\geq 0$, and $\eta_G(x,y)=0$ if and only if $x=y$,\\
\emph{(2)} Symmetry: $\eta_G(x,y)=\eta_G(y,x)$,\\
\emph{(3)} Triangle inequality: $\eta_G(x,y)\leq \eta_G(x,z)+\eta_G(z,y).$
\end{definition}

Let $\eta_G$ be now some arbitrary metric. An open ball defined with it is $B_\eta(x,r)=\{y\in G\text{ }|\text{ }\eta_G(x,y)<r\}$ and the corresponding closed ball is $\overline{B}_\eta(x,r)=\{y\in G\text{ }|\text{ }\eta_G(x,y)\leq r\}$. Denote the sphere of these balls by $S_\eta(x,r)$. For Euclidean metric, these notations are $B^n(x,r)$, $\overline{B}^n(x,r)$ and $S^{n-1}(x,r)$, respectively, where $n$ is the dimension. In this paper, the unit ball $\B^n=B^n(0,1)$, the upper half-plane $\uhp^n=\{(x_1,...,x_n)\in\R^n\text{ }|\text{ }x_n>0\}$ and the open sector $S_\theta=\{x\in\C\text{ }|\text{ }0<\arg(x)<\theta\}$ with an angle $\theta\in(0,2\pi)$ will be commonly used as domains $G$. Note also that the unit vectors will be denoted by $\{e_1,...,e_n\}$.

Let us now define the metrics needed for a domain $G\subsetneq\R^n$. Denote the Euclidean distance between the points $x,y$ by $|x-y|$ and let $d_G(x)=\inf\{|x-z|\text{ }|\text{ }z\in\partial G\}$. Suppose that the $t$-metric is defined with the Euclidean distance so that 
\begin{align*}
t_G(x,y)=\frac{|x-y|}{|x-y|+d_G(x)+d_G(y)}  
\end{align*}
for all $x,y\in G$, if not otherwise specified.

The following hyperbolic type metrics will be considered:\newline
The triangular ratio metric: $s_G:G\times G\to[0,1]$,
\begin{align*}
s_G(x,y)=\frac{|x-y|}{\inf_{z\in\partial G}(|x-z|+|z-y|)}, 
\end{align*}
the $j^*_G$-metric: $j^*_G:G\times G\to[0,1],$
\begin{align*}
j^*_G(x,y)=\frac{|x-y|}{|x-y|+2\min\{d_G(x),d_G(y)\}},    
\end{align*}
and the point pair function: $p_G:G\times G\to[0,1],$
\begin{align*}
p_G(x,y)=\frac{|x-y|}{\sqrt{|x-y|^2+4d_G(x)d_G(y)}}.   
\end{align*}
Out of these hyperbolic type metrics, the triangular ratio metric was studied by P. H\"ast\"o in 2002 \cite{h}, and the two other metrics are more recent. As pointed out in \cite{hvz}, the $j^*_G$-metric is derived from the \emph{distance ratio metric} found by F.W. Gehring and B.G. Osgood in \cite{GO79}. Note that there are proper domains $G$ in which the point pair function is not a metric \cite[Rmk 3.1 p. 689]{chkv}.

Define also the hyperbolic metric as
\begin{align*}
\text{ch}\rho_{\uhp^n}(x,y)&=1+\frac{|x-y|^2}{2d_{\uhp^n}(x)d_{\uhp^n}(y)},\quad x,y\in\uhp^n,\\
\text{sh}^2\frac{\rho_{\B^n}(x,y)}{2}&=\frac{|x-y|^2}{(1-|x|^2)(1-|y|^2)},\quad x,y\in\B^n
\end{align*}
in the upper half-plane $\uhp^n$ and in the Poincar\'e unit disk $\B^n$ \cite[(4.8), p. 52 \& (4.14), p. 55]{hkvbook}. In the two-dimensional space,
\begin{align*}
\text{th}\frac{\rho_{\uhp^2}(x,y)}{2}&=\text{th}\left(\frac{1}{2}\log\left(\frac{|x-\overline{y}|+|x-y|}{|x-\overline{y}|-|x-y|}\right)\right)=\left|\frac{x-y}{x-\overline{y}}\right|,\\
\text{th}\frac{\rho_{\B^2}(x,y)}{2}&=\text{th}\left(\frac{1}{2}\log\left(\frac{|1-x\overline{y}|+|x-y|}{|1-x\overline{y}|-|x-y|}\right)\right)=\left|\frac{x-y}{1-x\overline{y}}\right|=\frac{|x-y|}{A[x,y]},
\end{align*}
where $\overline{y}$ is the complex conjugate of $y$ and $A[x,y]=\sqrt{|x-y|^2+(1-|x|^2)(1-|y|^2)}$ is the Ahlfors bracket \cite[(3.17) p. 39]{hkvbook}.

Note that the following inequalities hold for the hyperbolic type metrics.

\begin{lemma}\label{jp_bounds}
\emph{\cite[Lemma 2.3, p. 1125]{hvz}} For a proper subdomain $G$ of $\R^n$, the inequality $j^*_G(x,y)\leq p_G(x,y)\leq\sqrt{2}j^*_G(x,y)$ holds for all $x,y\in G$.
\end{lemma}

\begin{lemma}\label{sj_properbounds}
\emph{\cite[Lemma 2.1, p. 1124 \& Lemma 2.2, p. 1125]{hvz}} For a proper subdomain $G$ of $\R^n$, the inequality $j^*_G(x,y)\leq s_G(x,y)\leq2j^*_G(x,y)$ holds for all $x,y\in G$.
\end{lemma}

\begin{lemma}\label{rhojsp_inG}
\emph{\cite[p. 460]{hkvbook}} For all $x,y\in G\in\{\uhp^n,\B^n\}$,
\begin{align*}
&(1)\quad{\rm th}\frac{\rho_{\uhp^n}(x,y)}{4}\leq j^*_{\uhp^n}(x,y)\leq s_{\uhp^n}(x,y)= p_{\uhp^n}(x,y)={\rm th}\frac{\rho_{\uhp^n}(x,y)}{2}\leq2{\rm th}\frac{\rho_{\uhp^n}(x,y)}{4},\\
&(2)\quad{\rm th}\frac{\rho_{\B^n}(x,y)}{4}\leq j^*_{\B^n}(x,y)\leq s_{\B^n}(x,y)\leq p_{\B^n}(x,y)\leq{\rm th}\frac{\rho_{\B^n}(x,y)}{2}\leq2{\rm th}\frac{\rho_{\B^n}(x,y)}{4}.    
\end{align*}
\end{lemma}

\section{t-Metric and Its Bounds}\label{sect_tbounds}

Now, we will prove that our new metric is truly a metric in the general case.

\begin{theorem}\label{thm_tismetric}
For any metric space $X$, a domain $G\subsetneq X$ and a metric $\eta_G$ defined in $G$, the function $t_G$ is a metric.
\end{theorem}
\begin{proof}
The function $t_G$ is a metric if it fulfills all the three conditions of Definition \ref{def_metric}. Trivially, the first two conditions hold. Consider now a function $f:[0,\infty)\to[0,\infty)$, $f(x)=x\slash(x+k)$ where $k>0$ is a constant. Since $f$ is increasing in its whole domain $[0,\infty)$, 
\begin{align*}
x\leq y\quad\Leftrightarrow\quad\frac{x}{x+k}\leq\frac{y}{y+k}. 
\end{align*}
Because $\eta_G$ is a metric, $\eta_G(x,y)\leq\eta_G(x,z)+\eta_G(z,y)$ for all $x,y,z\in G$. Furthermore, $\eta_G(z)\leq\min\{\eta_G(x,z)+\eta_G(x),\eta_G(z,y)+\eta_G(y)\}$. From these results, it follows that
\begin{align*}
t_G(x,y)&=\frac{\eta_G(x,y)}{\eta_G(x,y)+\eta_G(x)+\eta_G(y)}
\leq\frac{\eta_G(x,z)+\eta_G(z,y)}{\eta_G(x,z)+\eta_G(z,y)+\eta_G(x)+\eta_G(y)}\\
&=\frac{\eta_G(x,z)}{\eta_G(x,z)+\eta_G(z,y)+\eta_G(x)+\eta_G(y)}+\frac{\eta_G(z,y)}{\eta_G(x,z)+\eta_G(z,y)+\eta_G(x)+\eta_G(y)}\\
&\leq\frac{\eta_G(x,z)}{\eta_G(x,z)+\eta_G(x)+\eta_G(z)}+\frac{\eta_G(z,y)}{\eta_G(z,y)+\eta_G(y)+\eta_G(z)}
=t_G(x,z)+t_G(z,y)
\end{align*}
for all $x,y,z\in G$. Thus, $t_G$ fulfills the triangle inequality.
\end{proof}

We now show that the method of proof of Theorem \ref{thm_tismetric} can be used to prove that several other functions are metrics, too.

\begin{theorem}\label{thm_cmetric}
If $G$ is a proper subset of a metric space $X$, $\eta_G$ some metric defined in the closure of $G$ and $c_G:G\times G\to[0,\infty)$ some symmetric function such that, for all $x,y,z\in G$,
\begin{align}\label{ineq_c}
c_G(x,z)\leq\eta_G(z,y)+c_G(x,y),    
\end{align}
then any function $\phi_G:G\times G\to[0,1]$, defined as
\begin{align*}
\phi_G(x,x)=0,\quad
\phi_G(x,y)=\frac{\eta_G(x,y)}{\eta_G(x,y)+c_G(x,y)}\text{ if }x\neq y 
\end{align*}
for all $x,y\in G$, is a metric in the domain $G$.
\end{theorem}
\begin{proof}
Since $\eta_G$ is a metric and $c_G$ is both symmetric and non-negative, the function $\phi_G$ trivially fulfills the first two conditions of Definition \ref{def_metric}. Note that, by the triangle inequality of the metric $\eta_G$ and the inequality \eqref{ineq_c}, the inequalities
\begin{align*}
\eta_G(x,y)&\leq\eta_G(x,z)+\eta_G(z,y),\\
c_G(x,z)&\leq\eta_G(z,y)+c_G(x,y),\\
c_G(z,y)&\leq\eta_G(x,z)+c_G(x,y),
\end{align*}
hold for all $x,y,z\in G$. Now,
\begin{align*}
\phi_G(x,y)&=\frac{\eta_G(x,y)}{\eta_G(x,y)+c_G(x,y)}
\leq\frac{\eta_G(x,z)+\eta_G(z,y)}{\eta_G(x,z)+\eta_G(z,y)+c_G(x,y)}\\
&=\frac{\eta_G(x,z)}{\eta_G(x,z)+\eta_G(z,y)+c_G(x,y)}+
\frac{\eta_G(z,y)}{\eta_G(x,z)+\eta_G(z,y)+c_G(x,y)}\\
&\leq\frac{\eta_G(x,z)}{\eta_G(x,z)+c_G(x,z)}+
\frac{\eta_G(z,y)}{\eta_G(z,y)+c_G(z,y)}
=\phi_G(x,z)+\phi_G(z,y),
\end{align*}
so the function $\phi_G$ fulfills the triangle inequality and it must be a metric.
\end{proof}

\begin{remark}
(1) If the function $c_G$ of Theorem \ref{thm_cmetric} is strictly positive, the condition $\phi_G(x,x)=0$ does not need to be separately specified. Namely, this condition follows directly from the fact that $\eta_G(x,x)=0$ for a metric $\eta_G$. Note also that if $c_G$ is a null function, the function $\phi_G$ becomes the discrete metric.\newline 
(2) If $\eta_G$ is a metric, then $\eta_G^\alpha$ is a metric, too, for $0<\alpha\leq1$, but this is not true for $\alpha>1$ \cite[Ex. 5.24, p. 80]{hkvbook}.
\end{remark}

\begin{corollary}
The function $\psi:\B^n\times\B^n\to[0,1]$, defined as
\begin{align*}
\psi(x,x)=0,\quad\psi(x,y)=\frac{|x-y|}{|x-y|+c|x||y|}\text{ if }x\neq y,    
\end{align*}
for all $x,y\in\B^n$ with a constant $0<c\leq1$, is a metric on the unit ball.
\end{corollary}
\begin{proof}
Since now
\begin{align*}
c|x|(|z|-|y|)\leq|z|-|y|\leq|z-y|\quad\Rightarrow\quad
c|x||z|\leq|z-y|+c|x||y|,
\end{align*}
for all $x,y,z\in\B^n$, the result follows from Theorem \ref{thm_cmetric}.
\end{proof}

\begin{corollary}\label{cor_upsilon}
If $G$ is a proper subset of a metric space $X$ and $\eta_G$ is some metric defined in the closure of $G$ such that $\eta_G(x)=\inf\{\eta_G(x,u)\text{ }|\text{ }u\in \partial G\}\leq1$ for all $x\in G$, then a function $\upsilon_G:G\times G\to[0,1]$, defined as
\begin{align*}
\upsilon_G(x,y)=\frac{\eta_G(x,y)}{\eta_G(x,y)+c\sqrt{(1+\eta_G(x))(1+\eta_G(y))}}
\end{align*}
with a constant $0<c\leq\sqrt{2}$ is a metric in the domain $G$.
\end{corollary}
\begin{proof}
Fix $c_G(x,y)=c\sqrt{(1+\eta_G(x))(1+\eta_G(y))}$. Now,
\begin{align*}
c_G(x,z)-c_G(x,y)&=
c\sqrt{(1+\eta_G(x))(1+\eta_G(z))}-c\sqrt{(1+\eta_G(x))(1+\eta_G(y))}\\
&=c\cdot\frac{(1+\eta_G(x))(1+\eta_G(z))-(1+\eta_G(x))(1+\eta_G(y))}{\sqrt{(1+\eta_G(x))(1+\eta_G(z))}+\sqrt{(1+\eta_G(x))(1+\eta_G(y))}}\\
&=\frac{c\sqrt{1+\eta_G(x)}(\eta_G(z)-\eta_G(y))}{\sqrt{1+\eta_G(z)}+\sqrt{1+\eta_G(y)}}\leq\frac{\sqrt{2}c(\eta_G(z)-\eta_G(y))}{1+1}\\
&\leq\eta_G(z)-\eta_G(y)\leq\eta_G(z,y),
\end{align*}
so the inequality \eqref{ineq_c} holds for all $x,y,z\in G$ and the result follows from Theorem \ref{thm_cmetric}.
\end{proof}

\begin{corollary}
The function $\chi:\B^n\times\B^n\to[0,1]$, defined as
\begin{align*}
\chi(x,y)=\frac{|x-y|}{|x-y|+c\sqrt{(2-|x|)(2-|y|)}}    
\end{align*}
for all $x,y\in\B^n$ with a constant $0<c\leq\sqrt{2}$, is a metric on the unit ball.
\end{corollary}
\begin{proof}
Follows from Corollary \ref{cor_upsilon}.
\end{proof}

Let us focus again on the $t$-metric. Since the result of Theorem \ref{thm_tismetric} holds for any metric $\eta_G$, the $t$-metric is trivially a metric also when defined for the Euclidean metric. Below, we will consider the $t$-metric in this special case only. Let us next prove the inequalities between the $t$-metric and the three hyperbolic type metrics defined earlier.

\begin{theorem}\label{t.jps_bounds}
For all domains $G\subsetneq\R^n$ and all points $x,y\in G$, the following inequalities hold:\\
(1) $j^*_G(x,y)\slash2\leq t_G(x,y)\leq j^*_G(x,y)$,\\
(2) $p_G(x,y)\slash2\leq t_G(x,y)\leq p_G(x,y)$,\\
(3) $s_G(x,y)\slash2\leq t_G(x,y)\leq s_G(x,y)$.\\
Furthermore, in each case the constants are sharp for some domain $G$.
\end{theorem}
\begin{proof} (1) $t_G(x,y)\leq j^*_G(x,y)$ follows trivially from the definitions of these metrics. If $d_G(x)\leq d_G(y)$, then $d_G(y)\leq|x-y|+d_G(x)$ and $d_G(x)+d_G(y)\leq|x-y|+2d_G(x)$. In the same way, if $d_G(y)\leq d_G(x)$, then $d_G(x)+d_G(y)\leq|x-y|+2d_G(y)$. It follows from this that
\begin{align*}
d_G(x)+d_G(y)\leq|x-y|+2\min\{d_G(x),d_G(y)\}.    
\end{align*}
With this information, we can write
\begin{align*}
j^*_G(x,y)&=\frac{|x-y|}{|x-y|+2\min\{d_G(x),d_G(y)\}}\\
&=\frac{2|x-y|}{|x-y|+2\min\{d_G(x),d_G(y)\}+|x-y|+2\min\{d_G(x),d_G(y)\}}\\
&\leq\frac{2|x-y|}{|x-y|+2\min\{d_G(x),d_G(y)\}+d_G(x)+d_G(y)}\\
&\leq\frac{2|x-y|}{|x-y|+d_G(x)+d_G(y)}=2t_G(x,y).
\end{align*}
The equality $t_G(x,y)=j^*_G(x,y)$ holds always when $d_G(x)=d_G(y)$. For $x=i$ and $y=ki$, $\lim_{k\to0^+}(t_{\uhp^2}(x,y)\slash j^*_{\uhp^2}(x,y))=1\slash2$.
Thus, the first inequality of the theorem and its sharpness follow.\\
\\
(2) From Lemma \ref{jp_bounds} and Theorem \ref{t.jps_bounds}(1), it follows that $t_G(x,y)\leq p_G(x,y)$.

Let us now prove that $p_G(x,y)\slash2\leq t_G(x,y)$. This is clearly equivalent to
\begin{align}\label{ineq_tp}
|x-y|+d_G(x)+d_G(y)\leq2\sqrt{|x-y|^2+4d_G(x)d_G(y)}.    
\end{align}
Fix $u=|x-y|$, $v=\min\{d_G(x),d_G(y)\}$ and $k=|d_G(x)-d_G(y)|$. The inequality \eqref{ineq_tp} is now
\begin{align*}
u+2v+k\leq2\sqrt{u^2+4v(v+k)}\quad\Leftrightarrow\quad
k^2+4uv+2uk-3u^2-12v^2-12kv\leq0.
\end{align*}
Define a function $f(k)=k^2+4uv+2uk-3u^2-12v^2-12kv$. Since the inequality above is equivalent to $f(k)\leq0$, we need to find out the greatest value of this function. There is no upper limit for $u\geq0$ or $v\geq0$ but $0\leq k\leq u$. We can solve that
\begin{align*}
f'(k)=2k+2u-12v=0\quad\Leftrightarrow\quad k=6v-u.
\end{align*}
Since
\begin{align*}
f(0)&=-3u^2+4uv-12v^2\leq-2u^2-8v^2\leq0,\\
f(6v-u)&=-4u^2+16uv-42v^2\leq-26v^2\leq0,\\
f(u)&=-8uv-12v^2\leq0,
\end{align*}
$f(k)$ is always non-positive on the closed interval $k\in[0,u]$ and, consequently, the inequality $p_G(x,y)\slash2\leq t_G(x,y)$ follows.

For $x=ki$ and $y=i$, $\lim_{k\to0^+}(t_{\uhp^2}(x,y)\slash p_{\uhp^2}(x,y))=1\slash2$ and, for $x=ki$ and $y=1+ki$, $\lim_{k\to0^+}(t_{\uhp^2}(x,y)\slash p_{\uhp^2}(x,y))=1$.\\ 
\\ 
(3) By the triangle inequality and Lemma \ref{sj_properbounds} and Theorem \ref{t.jps_bounds}(1),
\begin{align*}
\frac{s_G(x,y)}{2}&=\frac{|x-y|}{\inf_{z\in\partial G}(|x-z|+|z-y|)+\inf_{z\in\partial G}(|x-z|+|z-y|)}\\
&\leq\frac{|x-y|}{\inf_{z\in\partial G}(|x-z|+|z-y|)+d_G(x)+d_G(y)}
\leq t_G(x,y)\\
&\leq j^*_G(x,y)
\leq s_G(x,y).    
\end{align*}
For $x=ki$ and $y=i$, $\lim_{k\to0^+}(t_{\uhp^2}(x,y)\slash s_{\uhp^2}(x,y))=1\slash2$ and, for $x=ki$ and $y=1+ki$, $\lim_{k\to0^+}(t_{\uhp^2}(x,y)\slash s_{\uhp^2}(x,y))=1$.
\end{proof}

\begin{proposition}
For any fixed domain $G\subsetneq\R^n$, the inequalities of Theorem \ref{t.jps_bounds} are sharp.
\end{proposition}
\begin{proof}
If $G$ is a proper subdomain, there must exist some ball $B^n(x,r)\subset G$ with $S^{n-1}(x,r)\cap\partial G\neq\varnothing$ where $r>0$. Fix $z\in S^{n-1}(x,r)\cap\partial G$ and $y\in[x,z]$ so that $|y-z|=kr$ with $k\in(0,1)$. Clearly, $d_G(x)=r$, $d_G(y)=kr$, $|x-y|=(1-k)r$ and $\inf_{z\in\partial G}(|x-z|+|z-y|)=1+k$. Consequently,
\begin{align*}
t_G(x,y)&=\frac{(1-k)r}{(1-k)r+r+kr}=\frac{1-k}{2},\\
j^*_G(x,y)&=p_G(x,y)=s_G(x,y)=\frac{1-k}{1+k}.
\end{align*}
It follows that
\begin{align*}
\lim_{k\to0^+}\frac{t_G(x,y)}{j^*_G(x,y)}&=
\lim_{k\to0^+}\frac{t_G(x,y)}{j^*_G(x,y)}=
\lim_{k\to0^+}\frac{t_G(x,y)}{j^*_G(x,y)}=
\lim_{k\to0^+}\left(\frac{1+k}{2}\right)=
\frac{1}{2},\\
\lim_{k\to1^-}\frac{t_G(x,y)}{j^*_G(x,y)}&=
\lim_{k\to1^-}\frac{t_G(x,y)}{j^*_G(x,y)}=
\lim_{k\to1^-}\frac{t_G(x,y)}{j^*_G(x,y)}=
\lim_{k\to1^-}\left(\frac{1+k}{2}\right)=1.
\end{align*}
Thus, regardless of how $G$ is chosen, the inequalities of Theorem \ref{t.jps_bounds} are sharp.
\end{proof}

Next, we will study the connection between the $t$-metric and the hyperbolic metric.

\begin{theorem}\label{rhot}
For all $x,y\in G\in\{\uhp^n,\B^n\}$, the inequality 
\begin{align*}
\frac{1}{2}{\rm th}\frac{\rho_G(x,y)}{2}\leq t_G(x,y)\leq{\rm th}\frac{\rho_G(x,y)}{2}
\end{align*}
holds and the constants here are sharp.
\end{theorem}
\begin{proof}
In the case $G=\uhp^n$, the result follows directly from Lemma \ref{rhojsp_inG}(1) and Theorem \ref{t.jps_bounds}(3).

Suppose now that $G=\B^n$. From Lemma \ref{rhojsp_inG}(2) and Theorem \ref{t.jps_bounds}(1), it follows that 
\begin{align*}
t_{\B^n}(x,y)\leq j^*_{\B^n}(x,y)\leq
\text{th}\frac{\rho_{\B^n}(x,y)}{2}. 
\end{align*}
Thus, the inequality $t_{\B^n}(x,y)\leq\text{th}(\rho_{\B^n}(x,y)\slash2)$ holds for all $x,y\in\B^n$ and the sharpness follows because $x=ke_1$ and $y=-ke_1$ fulfill $\lim_{k\to1^-}(t_{\B^n}(x,y)\slash\text{th}(\rho_{\B^n}(x,y)\slash2))=1$.

Since the values of the $t$-metrics and the hyperbolic metric in the domain $\B^n$ only depend on how the points $x,y$ are located on the two-dimensional plane fixed by these two points and the origin, we can assume without loss of generality that $n=2$. Consider now the quotient
\begin{align*}
\frac{t_{\B^2}(x,y)}{\text{th}(\rho_{\B^2}(x,y)\slash2)}=
\frac{A[x,y]}{|x-y|+1-|x|+1-|y|}.
\end{align*}
By \cite[Lemma 7.57.(1) p. 152]{avv}, $A[x,y]\geq|x-y|+(1-|x|)(1-|y|)$ for all $x,y\in\B^2$. It follows that
\begin{align*}
\frac{t_{\B^2}(x,y)}{\text{th}(\rho_{\B^2}(x,y)\slash2)}\geq
\frac{|x-y|+1-|x|-|y|+|x||y|}{|x-y|+1-|x|+1-|y|}\geq
\frac{|x-y|-|x|-|y|+1}{|x-y|-|x|-|y|+2}\geq\frac{1}{2}, 
\end{align*}
which proves that $\text{th}(\rho_{\B^2}(x,y)\slash2)\slash2\leq t_{\B^2}(x,y)$. By the observation above, this also holds in the more general case where $n$ is not fixed. The inequality is sharp, too: For $x=ke_1$ and $y=-ke_1$, $\lim_{k\to0^+}(t_{\B^n}(x,y)\slash\text{th}(\rho_{\B^n}(x,y)\slash2))=1\slash2$.
\end{proof}

\begin{theorem}
For a fixed angle $\theta\in(0,2\pi)$ and for all $x,y\in S_\theta$, the following results hold:\newline 
(1) $t_{S_\theta}(x,y)\leq
{\rm th}(\rho_{S_\theta}(x,y)\slash2)\leq
2(\pi\slash\theta)\sin(\theta\slash2) t_{S_\theta}(x,y)$ if $\theta\in(0,\pi)$,\newline
(2) $t_{S_\theta}(x,y)\leq{\rm th}(\rho_{S_\theta}(x,y)\slash2)\leq2t_{S_\theta}(x,y)$ if $\theta=\pi$,\newline
(3) $(\pi\slash\theta)t_{S_\theta}(x,y)\leq
{\rm th}(\rho_{S_\theta}(x,y)\slash2)\leq
 2t_{S_\theta}(x,y)$ if $\theta\in(\pi,2\pi)$.
\end{theorem}
\begin{proof}
Follows from Theorems \ref{t.jps_bounds}(3) and \ref{rhot}, and \cite[Cor. 4.9, p. 9]{sqm}.
\end{proof}

\section{Quasiconformal Mappings and Lipschitz Constants}\label{sect_qmaps}

In this section, we will study the behaviour of the $t$-metric under different conformal and quasiconformal mappings in order to demonstrate how this metric works.

\begin{remark}
The $t$-metric is invariant under all similarity maps. In particular, the $t$-metric defined in a sector $S_\theta$ is invariant under a reflection over the bisector of the sector and a stretching $x\mapsto r\cdot x$ with any $r>0$. Consequently, this allows us to make certain assumptions when choosing the points $x,y\in S_\theta$. 
\end{remark}

First, let us study how the $t$-metric behaves under a certain conformal mapping between two sectors with angles at most $\pi$.

\begin{lemma}
If $\alpha,\beta\in(0,\pi]$ and $f:S_\alpha\to S_\beta$, $f(z)=z^{(\beta\slash \alpha)}$, then for all $x,y\in S_\alpha$
\begin{align*}
&\frac{t_{S_\alpha}(x,y)}{2}\leq t_{S_\beta}(f(x),f(y))\leq\frac{\beta\sin(\alpha\slash2)}{\alpha\sin(\beta\slash2)}t_{S_\alpha}(x,y)\text{ if }\alpha\leq\beta,\\
&\frac{\beta\sin(\alpha\slash2)}{\alpha\sin(\beta\slash2)}t_{S_\alpha}(x,y)\leq t_{S_\beta}(f(x),f(y))\leq 2t_{S_\alpha}(x,y)\text{ otherwise.}
\end{align*}
\end{lemma}
\begin{proof}
If $\alpha\leq\beta$, by Theorem \ref{t.jps_bounds}(3) and \cite[Lemma 5.11, p. 13]{sqm},
\begin{align}\label{tfalphabeta_ine}
t_{S_\beta}(f(x),f(y))\geq \frac{s_{S_\beta}(f(x),f(y))}{2}\geq \frac{s_{S_\alpha}(x,y)}{2}\geq\frac{t_{S_\alpha}(x,y)}{2}.  
\end{align}

Suppose that $\alpha\leq\beta$ still. Fix $x=e^{hi}$ and $y=re^{ki}$, where $0<h\leq k<\alpha$ and $r\geq1$ without loss of generality. Consider the quotient
\begin{align}\label{quo_tbetaf}
\frac{t_{S_\beta}(f(x),f(y))}{t_{S_\alpha}(x,y)}
=\frac{|1-(re^{(k-h)i})^{\frac{\beta}{\alpha}}|(|1-re^{(k-h)i}|+\sin(\gamma)+r\sin(\mu))}{|1-re^{(k-h)i}|(|1-(re^{(k-h)i})^{\frac{\beta}{\alpha}}|+\sin(\frac{\beta}{\alpha}\gamma)+r^{\frac{\beta}{\alpha}}\sin(\frac{\beta}{\alpha}\mu))},
\end{align}
where $\mu=\min\{h,\alpha-h\}$ and $\gamma=\min\{k,\alpha-k\}$. This quotient is strictly decreasing with respect to $r$ and, since $r\geq1$, it attains its maximum value when $r=1$. Consequently, the quotient \eqref{quo_tbetaf} has an upper limit of
\begin{align}\label{quo_uppbeta}
\frac{\sin(\frac{\beta}{2\alpha}(k-h))(2\sin(\frac{k-h}{2})+\sin(\gamma)+\sin(\mu))}{\sin(\frac{k-h}{2})(2\sin(\frac{\beta}{2\alpha}(k-h))+\sin(\frac{\beta}{\alpha}\gamma)+\sin(\frac{\beta}{\alpha}\mu))}.    
\end{align}
The value of the quotient above is at greatest, when $k-h$ is at minimum and both $\gamma$ and $\mu$ are at maximum. This happens when $h<\alpha\slash2$ and $k=\alpha-h$. Now, $\gamma=\mu=h$ and the quotient \eqref{quo_uppbeta} is
\begin{align*}
\frac{\sin(\frac{\beta}{2\alpha}(\alpha-2h))(\sin(\frac{\alpha}{2}-h)+\sin(h))}{\sin(\frac{\alpha}{2}-h)(\sin(\frac{\beta}{2\alpha}(\alpha-2h))+\sin(\frac{\beta}{\alpha}h))}.    
\end{align*}
Since the expression above is strictly increasing with respect to $h$ and $h<\alpha\slash2$, the maximum value of the quotient \eqref{quo_tbetaf} is
\begin{align}\label{tf_lim}
\lim_{h\to\frac{\alpha}{2}^-}\left(\frac{\sin(\frac{\beta}{2\alpha}(\alpha-2h))(\sin(\frac{\alpha}{2}-h)+\sin(h))}{\sin(\frac{\alpha}{2}-h)(\sin(\frac{\beta}{2\alpha}(\alpha-2h))+\sin(\frac{\beta}{\alpha}h))}\right)=\frac{\beta\sin(\alpha\slash2)}{\alpha\sin(\beta\slash2)},
\end{align}
which, together with the inequality \eqref{tfalphabeta_ine}, proves the first part of our theorem. Suppose next that $\alpha>\beta$ instead. It can be now proved that the minimum value of the quotient \eqref{quo_tbetaf} is the same limit value \eqref{tf_lim} and, by Theorem \ref{t.jps_bounds}(3) and \cite[Lemma 5.11, p. 13]{sqm}, $t_{S_\beta}(f(x),f(y))\leq2t_{S_\alpha}(x,y)$. Thus, the theorem follows.
\end{proof}

Let us now consider a more general result than the one above. Namely, instead of studying a conformal power mapping, we can assume that, for domains $G_1,G_2\subset\R^2$, the mapping $f:G_1\to G_2=f(G_1)$ is a $K$-quasiconformal homeomorphism, see \cite[Ch. 2]{v71}. Let $c(K)$ be as in \cite[Thm 16.39, p. 313]{hkvbook}. Now, $c(K)\geq K$ and $c(K)\to1$ whenever $K\to1$. See also the book \cite{gh} by F.W. Gehring and K. Hag.

\begin{theorem}
If $\alpha,\beta\in(0,2\pi)$ and $f:S_\alpha\to S_\beta=f(S_\alpha)$ is a $K$-quasiconformal homeomorphism, the following inequalities hold for all $x,y\in S_\alpha$.
\begin{align*}
&(1)\quad 
\frac{\beta}{2c(K)^K\pi\sin(\beta\slash2)}t_{S_\alpha}(x,y)^K\leq t_{S_\beta}(f(x),f(y))\leq c(K)(\frac{\pi}{\alpha}\sin(\frac{\alpha}{2}))^{1\slash K}t_{S_\alpha}(x,y)^{1\slash K}\\
&\quad\quad\text{if }\alpha,\beta\in(0,\pi],\\
&(2)\quad
\frac{1}{2c(K)^K}t_{S_\alpha}(x,y)^K\leq t_{S_\beta}(f(x),f(y))\leq \frac{c(K)\pi}{\beta}(\frac{\pi}{\alpha}\sin(\frac{\alpha}{2}))^{1\slash K}t_{S_\alpha}(x,y)^{1\slash K}\\
&\quad\quad\text{if }\alpha\in(0,\pi)\text{ and }\beta\in(\pi,2\pi),\\
&(3)\quad
\frac{1}{2}(\frac{\alpha}{c(K)\pi})^Kt_{S_\alpha}(x,y)^K\leq t_{S_\beta}(f(x),f(y))\leq \frac{c(K)\pi}{\beta}t_{S_\alpha}(x,y)^{1\slash K}\\
&\quad\quad\text{if }\alpha,\beta\in[\pi,2\pi).
\end{align*}
\end{theorem}
\begin{proof}
Follows from Theorem \ref{t.jps_bounds}(3) and \cite[Cor. 5.7, p. 12]{sqm}.
\end{proof}

Next, we will focus on the radial mapping, which is another example of a quasiconformal mapping, see \cite[16.2, p. 49]{v71}.

\begin{theorem}
If $f:G\to G$ with $G=\B^2\backslash\{0\}$ is the radial mapping defined as $f(z)=|z|^{a-1}z$ for some $0<a<1$, then for all $x,y\in G$ such that $|x|=|y|$, the sharp inequality
\begin{align*}
t_{\B^2\backslash\{0\}}(x,y)\leq t_{\B^2\backslash\{0\}}(f(x),f(y))\leq\frac{1}{2^a-1}t_{\B^2\backslash\{0\}}(x,y)
\end{align*}
holds.
\end{theorem}
\begin{proof}
Fix $x=re^{ki}$ and $y=re^{-ki}$ with $0<r<1$ and $0<k<\pi\slash2$. Now, $f(x)=r^ae^{ki}$ and $f(y)=r^ae^{-ki}$. Consider the quotient
\begin{align}\label{tfb0_quo}
\frac{t_G(f(x),f(y))}{t_G(x,y)}=\frac{|f(x)-f(y)|(|x-y|+d_G(x)+d_G(y))}{|x-y|(|f(x)-f(y)|+d_G(f(x))+d_G(f(y)))},    
\end{align}
where
\begin{align*}
&|x-y|=2r\sin(k),\\
&|f(x)-f(y)|=2r^a\sin(k),\\
&d_G(x)=d_G(y)=\min\{r,1-r\},\\
&d_G(f(x))=d_G(f(y))=\min\{r^a,1-r^a\}.
\end{align*}

If $0<r<r^a<1\slash2$, the quotient \eqref{tfb0_quo} is
\begin{align*}
\frac{t_G(f(x),f(y))}{t_G(x,y)}=\frac{r^a(r\sin(k)+r)}{r(r^a\sin(k)+r^a)}=1.    
\end{align*}

If $0<r\leq1\slash2<r^a<1$, the quotient \eqref{tfb0_quo} is
\begin{align}\label{tfb0_quo_0}
\frac{t_G(f(x),f(y))}{t_G(x,y)}=\frac{r^a(r\sin(k)+r)}{r(r^a\sin(k)+1-r^a)},   
\end{align}
which is decreasing with respect to $k$. Since
\begin{align*}
\lim_{k\to0^+}\left(\frac{r^a(r\sin(k)+r)}{r(r^a\sin(k)+1-r^a)}\right)=\frac{r^{1+a}}{r(1-r^a)}=\frac{r^a}{1-r^a}    
\end{align*}
is increasing with respect to $r$ and $r\leq1\slash2$, the maximum value of the quotient \eqref{tfb0_quo_0} is $1\slash(2^a-1)$. The other limit value
\begin{align*}
\lim_{k\to1^-}\left(\frac{r^a(r\sin(k)+r)}{r(r^a\sin(k)+1-r^a)}\right)=\frac{r^a(r+r)}{r(r^a+1-r^a)}=2r^a    
\end{align*}
is increasing with respect to $r^a$ and $r^a>1\slash2$, so the quotient \eqref{tfb0_quo_0} is always more than 1.

If $1\slash2<r<r^a<1$, the quotient \eqref{tfb0_quo} is
\begin{align}\label{tfb0_quo_1}
\frac{t_{\B^2\backslash\{0\}}(f(x),f(y))}{t_{\B^2\backslash\{0\}}(x,y)}=\frac{r^a(r\sin(k)+1-r)}{r(r^a\sin(k)+1-r^a)},   
\end{align}
which is decreasing with respect to $k$. Since $r>1\slash2$ and
\begin{align*}
\lim_{k\to0^+}\left(\frac{r^a(r\sin(k)+1-r)}{r(r^a\sin(k)+1-r^a)}\right)=\frac{r^a(1-r)}{r(1-r^a)},   
\end{align*}
is decreasing with respect to $r$, the quotient \eqref{tfb0_quo_0} is less than $1\slash(2^a-1)$. The other limit value is
\begin{align*}
\lim_{k\to1^-}\left(\frac{r^a(r\sin(k)+1-r)}{r(r^a\sin(k)+1-r^a)}\right)=r^{a-1},    
\end{align*}
which is clearly more than 1. 

Thus, the minimum value of the quotient \eqref{tfb0_quo} is 1 and the maximum value $1\slash(2^a-1)$, so the theorem follows.
\end{proof}

Let us now find Lipschitz constants of a few different mappings for the $t$-metric. 

\begin{theorem}\label{thm_tmobius}
For all conformal mappings $f:G_1\to G_2=f(G_1)$ with $G_1,G_2\in\{\uhp^n,\B^n\}$, the inequality 
\begin{align*}
\frac{1}{2}t_{G_1}(x,y)
\leq t_{G_2}(f(x),f(y))
\leq2t_{G_1}(x,y)    
\end{align*}
holds for all $x,y\in G_1$.
\end{theorem}
\begin{proof}
By Theorem \ref{rhot} and the conformal invariance of the hyperbolic metric,
\begin{align*}
\frac{1}{2}t_{G_1}(x,y)
&\leq\frac{1}{2}\text{th}\frac{\rho_{G_1}(x,y)}{2}
=\frac{1}{2}\text{th}\frac{\rho_{G_2}(f(x),f(y))}{2}
\leq t_{G_2}(f(x),f(y))\\
&\leq\text{th}\frac{\rho_{G_2}(f(x),f(y))}{2}
=\text{th}\frac{\rho_{G_1}(x,y)}{2}
\leq2t_{G_1}(x,y).
\end{align*}
\end{proof}

It follows from Theorem \ref{thm_tmobius} that the Lipschitz constant Lip$(f|G_1)$ for the $t$-metric in any conformal mapping $f:G_1\to G_2=f(G_1)$, $G_1,G_2\in\{\uhp^n,\B^n\}$, is at most 2. Suppose now that $h$ is the M\"obius transformation $h:\B^2\to\uhp^2$, $h(z)=(1-z)i\slash(1+z)$. Since, for $x=0$ and $y=\frac{1-k}{k+1}$ with $0<k<1$,
\begin{align*}
\lim_{k\to1^-}\left(\frac{t_{\uhp^n}(h(x),h(y))}{t_{\B^n}(x,y)}\right)=\lim_{k\to1^-}(k+1)=2,
\end{align*}
the Lipschitz constant Lip$(h|\B^2)$ is equal to 2. However, for certain M\"obius transformations, there might be a better constant than 2. For instance, the following conjecture is supported by several numerical tests.

\begin{conjecture}\label{conj_tmobius}
For all $a,x,y\in\B^2$, the Möbius transformation $T_a:\B^2\to\B^2$, $T_a(z)=(z-a)\slash(1-\overline{a}z)$ fulfills the inequality
\begin{align*}
t_{\B^2}(T_a(x),T_a(y))\leq(1+|a|)t_{\B^2}(x,y).
\end{align*}
\end{conjecture}

\begin{remark}\label{rmk_sbmobius}
It is also an open question whether the inequality of Conjecture \ref{conj_tmobius} holds for the triangular ratio metric or so called Barrlund metric, but numerical tests suggest so, see \cite[Conj. 1.6, p. 684]{chkv} and \cite[Conj. 4.3, p. 25]{fmv}.
\end{remark}

In the next few results, we will study a mapping $f^*:S_\theta\to S_\theta$, $f^*(x)=x\slash|x|^2$ defined in some open sector $S_\theta$, and find its Lipschitz constants for the $t$-metric.

\begin{theorem}
If $\theta\in(0,\pi]$ and $f^*$ is the mapping $f^*:S_\theta\to S_\theta$, $f^*(x)=x\slash|x|^2$, the Lipschitz constant Lip$(f^*|S_\theta)$ for the $t$-metric is $1+\sin(\theta\slash2)$.
\end{theorem}
\begin{proof}
Without loss of generality, we can fix $x=e^{hi}$ and $y=re^{ki}$ with $0<h\leq\pi\slash2$, $h\leq k<\theta$ and $r>0$. Since $x^*=e^{hi}$ and $y^*=(1\slash r)e^{ki}$, it follows that 
\begin{align*}
\frac{t_{S_\theta}(x^*,y^*)}{t_{S_\theta}(x,y)}
=\frac{\sqrt{1+r^2-2r\cos(k-h)}+r\sin(h)+\sin(\min\{k,\theta-k\})}{\sqrt{1+r^2-2r\cos(k-h)}+\sin(h)+r\sin(\min\{k,\theta-k\})}.
\end{align*}
To maximize this, we clearly need to choose $k=\theta\slash2$ and make $r$ and $h$ as small as possible. If $k=\theta\slash2$,
\begin{align*}
\lim_{h\to0^+,\text{ }r\to0^+}\frac{t_{S_\theta}(x^*,y^*)}{t_{S_\theta}(x,y)}
=1+\sin(\theta\slash2),
\end{align*}
so the theorem follows.
\end{proof}

\begin{theorem}\label{f*_for_s_inS}
If $x^*=x\slash|x|^2$ and $y^*=y\slash|y|^2$, the equality $s_{S_\theta}(x,y)=s_{S_\theta}(x^*,y^*)$ holds in an open sector $S_\theta$ with $\theta\in(0,2\pi)$.
\end{theorem}
\begin{proof}
Fix $x=e^{hi}$ and $y=re^{ki}$ where $r>0$ and $0<h\leq k<\theta$. Clearly, $x^*=x=e^{hi}$ and $y^*=(1\slash r)e^{ki}$. Suppose first that $\theta\leq\pi$. By the known solution to Heron's problem, the infimum $\inf_{z\in\partial S_\theta}(|x-z|+|z-y|)$ is $\min\{|\overline{x}-y|,|x-y'|\}$, where $y'$ is the point $y$ reflected over the left side of the sector $\theta$. Clearly,

\begin{align*}
&|\overline{x}-y|\leq|x-y'|
\quad\Leftrightarrow\quad
|e^{-hi}-re^{ki}|\leq|e^{hi}-re^{(2\theta-k)i}|\\
\Leftrightarrow\quad
&|1-re^{(h+k)i}|\leq|1-re^{(2\theta-h-k)i}|
\quad\Leftrightarrow\quad
h+k\leq2\theta-h-k
\quad\Leftrightarrow\quad
(h+k)\slash2\leq\theta\slash2.
\end{align*}

By symmetry, we can suppose that $(h+k)\slash2\leq\theta\slash2$ without loss of generality. Note that it follows from above that not only $\inf_{z\in\partial S_\theta}(|x-z|+|z-y|)=|\overline{x}-y|$ but also
$\inf_{z\in\partial S_\theta}(|x^*-z|+|z-y^*|)=|\overline{x^*}-y^*|$. Now,
\begin{align*}
s_{S_\theta}(x,y)&=\frac{|x-y|}{|\overline{x}-y|}=\frac{|1-re^{(k-h)i}|}{|1-re^{(k+h)i}|}=\frac{|1-re^{(h-k)i}|}{|1-re^{-(k+h)i}|}
=\frac{|r-e^{(k-h)i}|}{|r-e^{(k+h)i}|}\\
&=\frac{|1-(1\slash r)e^{(k-h)i}|}{|1-(1\slash r)e^{(k+h)i}|}
=\frac{|x^*-y^*|}{|\overline{x^*}-y^*|}=s_{S_\theta}(x^*,y^*).
\end{align*}

Consider now the case where $\theta>\pi$. If $k-h\geq\pi$, then $s_{S_\theta}(x,y)=1=s_{S_\theta}(x^*,y^*)$ always so suppose that $k-h<\pi$ instead. This leaves us three possible options. If $(h+k)\slash2\leq\pi\slash2$, then $(h+k)\slash2<\theta\slash2$ and
\begin{align*}
s_{S_\theta}(x,y)&=\frac{|x-y|}{|\overline{x}-y|}=\frac{|x^*-y^*|}{|\overline{x^*}-y^*|}=s_{S_\theta}(x^*,y^*),
\end{align*}
just like above. By symmetry, $s_{S_\theta}(x,y)=s_{S_\theta}(x^*,y^*)$ also if $(k+h)\slash2\geq\theta-\pi\slash2$. If $\pi\slash2<(k+h)\slash2<\theta-\pi\slash2$ instead, then
\begin{align*}
s_{S_\theta}(x,y)&=\frac{|x-y|}{|x|+|y|}=\frac{|1-re^{(k-h)i}|}{1+r}=\frac{|r-e^{(k-h)i}|}{1+r}=\frac{|1-(1\slash r)e^{(k-h)i}|}{1+1\slash r}\\
&=\frac{|x^*-y^*|}{|x^*|+|y^*|}=s_{S_\theta}(x^*,y^*).
\end{align*}
\end{proof}

\begin{theorem}
If $\theta\in[\pi,2\pi)$ and $f^*$ is the mapping $f^*:S_\theta\to S_\theta$, $f^*(x)=x\slash|x|^2$, the Lipschitz constant Lip$(f^*|S_\theta)$ for the $t$-metric is 2.
\end{theorem}
\begin{proof}
It follows from Theorems \ref{f*_for_s_inS} and \ref{t.jps_bounds}(3) that 
\begin{align*}
\frac{t_{S_\theta}(x,y)}{2}
\leq t_{S_\theta}(x^*,y^*)
\leq 2t_{S_\theta}(x,y)
\end{align*}
for all $x,y\in S_\theta$. Since for $x=e^{hi}$ and $y=re^{\pi i\slash2}$ with $h<\pi\slash2$ and $r>0$,
\begin{align*}
\lim_{h,r\to0^+}\frac{t_{S_\theta}(x^*,y^*)}{t_{S_\theta}(x,y)}=
\lim_{h,r\to0^+}\left(\frac{\sqrt{1+r^2-2r\cos(\pi\slash2-h)}+r\sin(h)+1}{\sqrt{1+r^2-2r\cos(\pi\slash2-h)}+\sin(h)+r}\right)
=2,
\end{align*}
and it follows that 
\begin{align*}
\sup\{\frac{t_{S_\theta}(x^*,y^*)}{t_{S_\theta}(x,y)}\text{ }|\text{ }x,y\in S_\theta,x\neq y,\theta\in[\pi,2\pi)\}=2.    
\end{align*}
\end{proof}

\section{Comparison of Metric Balls}\label{sect_metricballs}

Next, we will graphically demonstrate the differences and similarities between the various metrics considered in this paper by drawing for each metric several circles centered at the same point but with different radii. In all of the figures of this section, the domain $G\subset\R^2$ is a regular five-pointed star and the circles have a radius of $r=1\slash10,...,9\slash10$. The center of these circles is in the center of $G$ in the first figures, and then off the center in the rest of the figures. All the figures in this section were drawn by using the contour plot function \texttt{contour} in R-Studio and choosing a grid of the size 1,000$\times$1,000 test points. While we graphically only inspect circles and disks, we will also prove some properties for the $n$-dimensional metric balls.

\begin{figure}[!tbp]
  \centering
  \subfloat[$s_G$-metric circles.]{\includegraphics[width=0.49\textwidth]{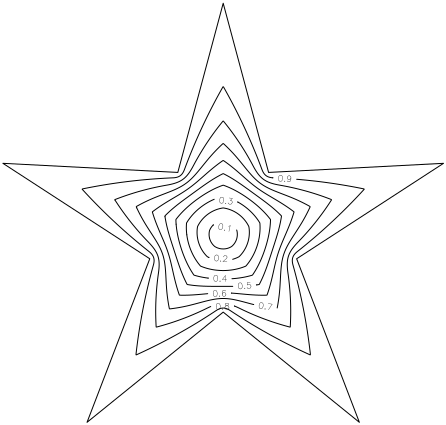}\label{fig4}}
  \hfill
  \subfloat[$j^*_G$-metric circles.]{\includegraphics[width=0.49\textwidth]{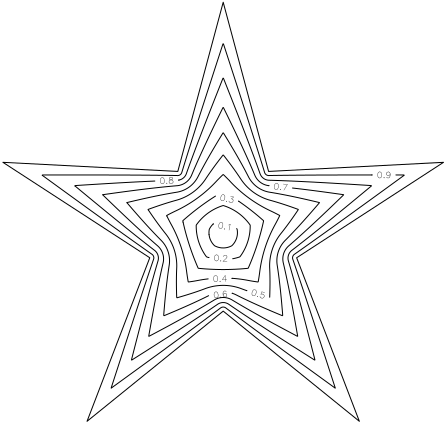}\label{fig5}}
  \\
  \subfloat[$p_G$-metric circles.]{\includegraphics[width=0.49\textwidth]{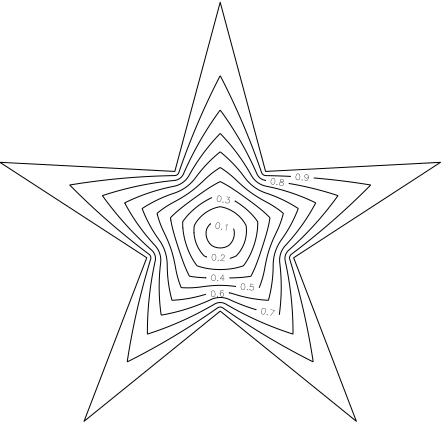}\label{fig6}}
  \hfill
  \subfloat[$t_G$-metric circles.]{\includegraphics[width=0.49\textwidth]{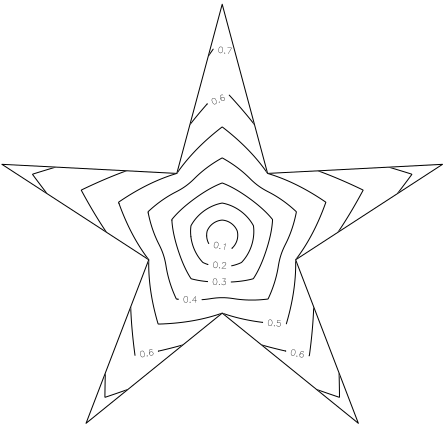}\label{fig7}}
  \caption{Circles in a five-pointed star domain with different metrics.}
\end{figure}

For several hyperbolic type metrics, the metric balls of small radii resemble Euclidean balls, but the geometric structure of the boundary of the domain begins to affect the shape of these balls when their radii grow large enough, see \cite[Ch. 13, pp. 239-259]{hkvbook}. By analysing this phenomenon more carefully, we can observe, for instance, that the balls are convex with radii less than some fixed $r_0>0$ in the case of some other metrics, see \cite[Thm 13.6, p. 241; Thm 13.41 p. 256; Thm 13.44, p. 258]{hkvbook}. From the figures of this section, we see that the four metrics studied here share this same property. In particular, we notice that, while the metric disks with small radii are convex and round like Euclidean disks, the metric circles with larger radii are non-convex and have corner points. By a corner point, we mean here such a point on the circle arc that has many possible tangents. 

In the following theorem, we will prove a property that can be seen from Figures \ref{fig6}, \ref{fig7}, \ref{fig10} and \ref{fig11}.

\begin{theorem}
If the domain $G$ is a polygon, then the corner points of the circles $S_p(x,r)$ and $S_t(x,r)$ are located on the the angle bisectors of $G$.
\end{theorem}
\begin{proof}
Suppose $G$ has sides $l_0$ and $l_1$ that have a common endpoint $k$. Fix $x\in G$ and choose some point $y\in G$ so that $k$ is the vertex of $G$ that is closest to $y$ and there is no other side closer to $y$ than $l_0$ and $l_1$. Thus, $d_G(y)=\min\{d(y,l_0),d(y,l_1)\}$ and, for a fixed distance $|x-y|$, $d_G(y)$ is at maximum when $d(y,l_0)=d(y,l_1)$. The condition $d(y,l_0)=d(y,l_1)$ is clearly fulfilled when $y$ is on the bisector of $\angle(l_0,l_1)$ and, the greater the $d_G(y)$, the smaller the distances $p_G(x,y)$ and $t_G(x,y)$ are now. Consequently, if the circle $S_p(x,r)$ or $S_t(x,r)$ has a corner point, it must be located on an angle bisector of $G$.
\end{proof}

\begin{figure}[!tbp]
  \centering
  \subfloat[$s_G$-metric circles.]{\includegraphics[width=0.49\textwidth]{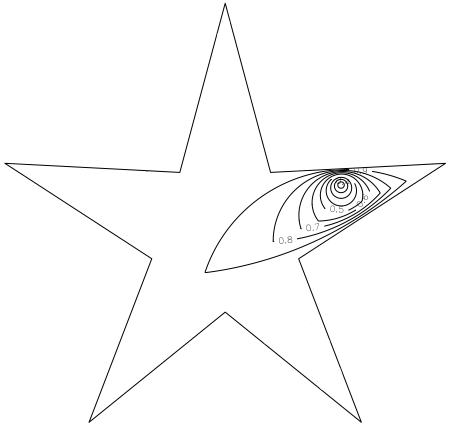}\label{fig8}}
  \hfill
  \subfloat[$j^*_G$-metric circles.]{\includegraphics[width=0.49\textwidth]{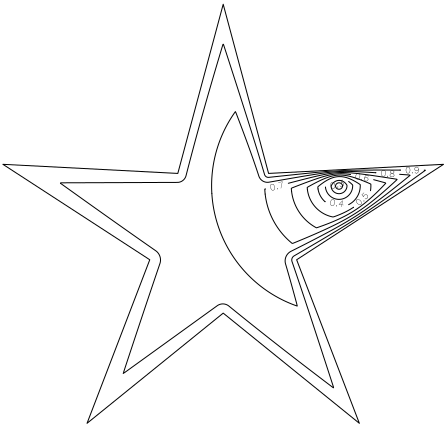}\label{fig9}}
  \\
  \subfloat[$p_G$-metric circles.]{\includegraphics[width=0.49\textwidth]{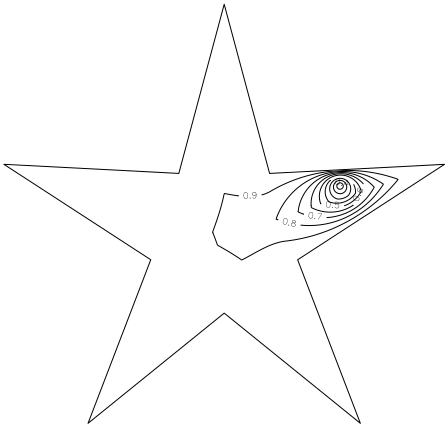}\label{fig10}}
  \hfill
  \subfloat[$t_G$-metric circles.]{\includegraphics[width=0.49\textwidth]{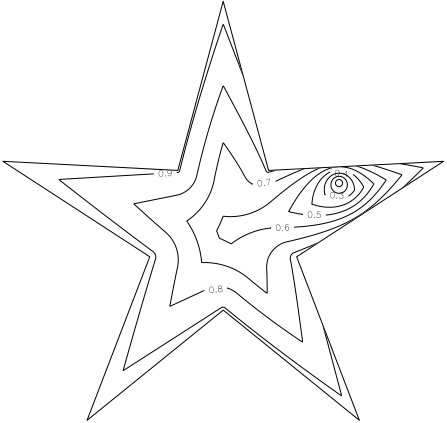}\label{fig11}}
  \caption{Circles in a five-pointed star domain with different metrics.}
\end{figure}

However, it can been seen from Figures \ref{fig8} and \ref{fig9} that the circles with $s_G$- and $j^*_G$-metrics can have corner points also elsewhere than on the angle bisectors of the domain $G$. We also notice that the circles in Figure \ref{fig9} clearly differ those in Figures \ref{fig8} and \ref{fig10}. This can be described with the concept of starlikeness, which is a looser form of convexity. Namely, a set $K$ is \emph{starlike} with respect to a point $x\in K$ if and only if the segment $[x,y]$ belongs to $K$ fully for every $y\in K$. In particular, the five-pointed star domain is starlike with respect to its center. The disks by the $j^*_G$- and $t_G$-metrics (Figures \ref{fig8} and \ref{fig10}) are clearly not starlike and, even if it cannot be clearly seen from Figure \ref{fig11}, there are disks drawn with the point pair function $p_G$ that are not starlike.

\begin{lemma}
There exist disks $B_{j^*}(x,r)$, $B_p(x,r)$, and $B_t(x,r)$ that are not starlike with respect to their center. 
\end{lemma}
\begin{proof}
Consider a domain $G=\uhp^2\cup\{z\in\C\text{ }|\text{ }-1<\text{Re}(z)<1,-3<\text{Im}(z)\leq0\}$. Fix $x=-2i$ and $y=3+i$. Clearly, $d_G(x)=d_G(y)=1$ and $|x-y|=3\sqrt{2}$. Consequently, 
\begin{align*}
j^*_G(x,y)=t_G(x,y)=\frac{3}{3+\sqrt{2}}<0.7,\quad p_G(x,y)=\frac{3}{\sqrt{11}}<0.91.    
\end{align*}
The segment $[x,y]$ does not clearly belong to $G$ fully and no disk in $G$ can contain this segment. However, its end point $y$ is clearly included in the disks $B_{j^*}(x,0.7)$, $B_t(x,0.7)$ and $B_p(x,0.91)$. Thus, we have found examples of non-starlike disks.
\end{proof}

There are no disks or balls like this for the triangular ratio metric. 

\begin{lemma}
\emph{\cite[p. 206]{hkvbook}} The balls $B_s(x,r)$ in any domain $G\subsetneq\R^n$ are always starlike with respect to their center $x$. 
\end{lemma}

For several common hyperbolic type metrics $\eta_G$, the closed ball $\overline{B_\eta(x,M)}$ with $M=\eta_G(x,y)$ and $x,y\in G$ is always a compact subset of the domain $G$, see \cite[p. 79]{hkvbook}. For instance, the hyperbolic metric $\rho_G$ has this property \cite[p. 192]{hkvbook}. As can be seen from the figures, the $j^*$-metric, the triangular ratio metric and the point pair function share this property, too.

\begin{lemma}\label{sjp_touchesboundary}
The balls $B_{j^*}(x,r)$, $B_p(x,r)$ and $B_s(x,r)$ touch the boundary of the domain $G\subsetneq\R^n$ if and only if $r=1$.
\end{lemma}
\begin{proof}
If the ball $B_\eta(x,r)$, $\eta_G\in\{j^*_G, p_G, s_G\}$, touches the boundary of $G$, then there is some point $y\in S_\eta(x,r)$ with $d_G(y)=0$ and $j^*_G(x,y)=p_G(x,y)=s_G(x,y)=1$. Thus, we need to just prove that the balls with radius 1 always touch the boundary. Consider first the balls $B_{j^*}(x,r)$ and $B_p(x,r)$, with a radius $r=1$. Since, for all the points $y$ on their boundary,
\begin{align*}
j^*_G(x,y)&=1\quad\Leftrightarrow\quad 2\min\{d_G(x),d_G(y)\}=0\quad\Leftrightarrow\quad d_G(x)=0\text{ or }d_G(y)=0,\\
p_G(x,y)&=1\quad\Leftrightarrow\quad 4d_G(x)d_G(y)=0\quad\Leftrightarrow\quad d_G(x)=0\text{ or }d_G(y)=0,
\end{align*}
the balls $B_{j^*}(x,1)$ and $B_p(x,1)$ touch the boundary of $G$.

Consider yet the triangular ratio metric. Because only balls with radius $r=1$ can touch the boundary, $B_s(x,1)\cap\partial G=\varnothing$. However, if $s_G(x,y)=1$, there is some point $z\in\partial G$ such that $|x-y|=|x-z|+|z-y|$. This means that $z$ is on a line segment $[x,y]$ and, since $z\notin B_s(x,1)$, $z$ must be arbitrarily close to the point $y$. Thus, $d_G(y)=0$ and the ball $B_s(x,1)$ touches the boundary.
\end{proof}

However, the $t$-metric differs from the hyperbolic type metrics in this aspect: the closure of a $t$-metric ball is a compact set, if and only if the radius of the ball is less than $1\slash2$.

\begin{theorem}\label{thm_ttouchesboundary}
The balls $B_t(x,r)$ touch the boundary of the domain $G\subsetneq\R^n$ if and only if $r\geq\frac{1}{2}$.
\end{theorem}
\begin{proof}
If $B_t(x,r)$ touches the boundary, there must be some $y\in S_t(x,r)$ such that $d_G(y)=0$. Since $d_G(x)\leq|x-y|+d_G(y)$, it follows that
\begin{align*}
r=t_G(x,y)=\frac{|x-y|}{|x-y|+d_G(x)+d_G(y)}\geq\frac{|x-y|}{|x-y|+|x-y|+0+0}=\frac{1}{2}.    
\end{align*}
Thus, only balls $B_t(x,r)$ with a radius $r\geq\frac{1}{2}$ can touch the boundary of $G$. 

Let us yet prove that the balls $B_t(x,\frac{1}{2})$ always touch the boundary of $G$. For any point $y\in S_t(x,\frac{1}{2})$, it holds that $|x-y|=1\slash2$ and
\begin{align*}
t_G(x,y)=\frac{|x-y|}{|x-y|+d_G(x)+d_G(y)}=\frac{1}{2}\quad
\Leftrightarrow\quad d_G(y)=|x-y|-d_G(x).
\end{align*}
Since only balls $B_t(x,r)$ with $r\geq\frac{1}{2}$ can touch the boundary of $G$, $B_t(x,\frac{1}{2})\cap\partial G=\varnothing$ and $d_G(x)\geq1\slash2$. Thus, $d_G(y)=1\slash2-d_G(x)\leq0$ and, since the distances cannot be negative, $d_G(y)=0$ and the ball $B_t(x,\frac{1}{2})$ truly touches the boundary of $G$. 
\end{proof}

The result above is visualized in Figures \ref{fig7} and \ref{fig11}.

\addcontentsline{toc}{section}{References} 
\renewcommand{\refname}{References} 



\begin{thebibliography}{10}

\bibitem{avv}{\sc
G. Anderson, M. Vamanamurthy and M. Vuorinen,}
\emph{Conformal Invariants, Inequalities, and Quasiconformal Maps.} Wiley-Interscience, 1997.

\bibitem{bm}{\sc  A.F. Beardon and D. Minda,}  \emph{The hyperbolic metric and geometric 
function theory,} Proc.   International Workshop on  Quasiconformal Mappings and their 
Applications (IWQCMA05), eds. S. Ponnusamy, T. Sugawa and M. Vuorinen (2006), 9-56.

\bibitem{chkv}{\sc
J. Chen, P. Hariri, R. Kl\'en and M. Vuorinen,}
Lipschitz conditions, triangular ratio metric, and quasiconformal maps.
\emph{Ann. Acad. Sci. Fenn. Math., 40} (2015), 683-709.

\bibitem{fmv}{\sc 
M. Fujimura, M. Mocanu and M. Vuorinen,} 
Barrlund's distance function and quasiconformal
maps, \emph{Complex Var. Elliptic Equ.} (2020), 1-31.

\bibitem{gh}{ \textsc{F.W. Gehring and  K. Hag},} \emph{The ubiquitous quasidisk.} With contributions
by Ole Jacob Broch. Mathematical Surveys and Monographs, 184. American Mathematical Society,
Providence, RI, 2012.

\bibitem{GO79}{\sc 
F.W. Gehring and B.G. Osgood,}
Uniform domains and the quasi-hyperbolic metric, \emph{J. Analyse Math., 36} (1979), 50-74.

\bibitem{hkvbook}{\sc
P. Hariri, R. Kl\'en and M. Vuorinen,}
\emph{Conformally Invariant Metrics and Quasiconformal Mappings.}
Springer, 2020.

\bibitem{hvz}{\sc
P. Hariri, M. Vuorinen and X. Zhang,}
Inequalities and Bilipschitz Conditions for Triangular Ratio Metric.
\emph{Rocky Mountain J. Math., 47}, 4 (2017), 1121-1148.

\bibitem{h}{\sc  
P. H\"ast\"o,} 
A new weighted metric, the relative metric I. \emph{J. Math. Anal. Appl., 274} (2002), 38-58.

\bibitem{imsz}{\sc  
Z. Ibragimov, M. Mohapatra, S. Sahoo and X. Zhang,} 
Geometry of the Cassinian metric and its inner metric. \emph{Bull. Malays. Math. Sci. Soc., 40} (2017), no. 1, 361-372.

\bibitem{ms}{\sc
M. Mohapatra and S. Sahoo,}
A Gromov hyperbolic metric vs the hyperbolic and other related metrics. (English summary)
\emph{Comput. Methods Funct. Theory, 18} (2018), no. 3, 473-493.

\bibitem{sqm}{\sc
O. Rainio and M. Vuorinen,}
Triangular Ratio Metric Under Quasiconformal Mappings In Sector Domains. Arxiv, 2005.11990.

\bibitem{v71}{\sc
J. V\"ais\"al\"a,}
Lectures on $n$-dimensional quasiconformal mappings. Lecture Notes in Math. Vol. 229, Springer-Verlag, Berlin- Heidelberg- New York, 1971.

\end{thebibliography}
\end{document}